\documentclass[a4paper,12pt]{article}
\title{Low analytic rank implies low partition rank for tensors}

\usepackage{amsmath, wrapfig}
\usepackage{dsfont, a4wide, amsthm, amssymb, amsfonts, graphicx}
\usepackage{fancyhdr, xspace, psfrag, setspace, supertabular, color}
\usepackage{hyperref}
\hypersetup{
     colorlinks=true}
\usepackage{bbm}
\usepackage{stmaryrd}
\usepackage{txfonts}
\usepackage{enumitem}

\newtheorem{theorem}{Theorem}[section]

\newtheorem{lemma}[theorem]{Lemma}

\newtheorem{definition}[theorem]{Definition}
\newtheorem*{definition*}{Definition}

\theoremstyle{remark}
\newtheorem{remark}[theorem]{Remark}

\def\e{\epsilon}
\def\E{\mathbb{E}}

\def\C{\mathbb{C}}

\def\P{\mathbb{P}}
\def\F{\mathbb{F}}

\def\a{\alpha}

\def\d{\delta}

\def\b1{\mathbbm{1}}

\def \cB{\mathcal B}
\def \cC{\mathcal C}
\def \cD{\mathcal D}
\def \cG{\mathcal G}

\def\rank{\mathop{\mathrm{rank}}}

\def \tower{{\rm tower}}
\def \bias{{\rm bias}}
\def \arank{{\rm arank}}
\def \prank{{\rm prank}}

\def \s{\subset}

\begin{document}
	
\author{Oliver Janzer\thanks{Department of Pure Mathematics and Mathematical Statistics, University of Cambridge. E-mail: oj224@cam.ac.uk}}

\date{}	
	
\maketitle

\begin{abstract}
A tensor defined over a finite field $\F$ has low analytic rank if the distribution of its values differs significantly from the uniform distribution. An order $d$ tensor has partition rank 1 if it can be written as a product of two tensors of order less than $d$, and it has partition rank at most $k$ if it can be written as a sum of $k$ tensors of partition rank 1. In this paper, we prove that if the analytic rank of an order $d$ tensor is at most $r$, then its partition rank is at most $f(r,d,|\F|)$. Previously, this was known with $f$ being an Ackermann-type function in $r$ and $d$ but not depending on $\F$. The novelty of our result is that $f$ has only tower-type dependence on its parameters. It follows from our results that a biased polynomial has low rank; there too we obtain a tower-type dependence improving the previously known Ackermann-type bound.
\end{abstract}

\section{Introduction}

\subsection{Bias and rank of polynomials}

For a polynomial $P: \F^n\rightarrow \F$, we say that $P$ is unbiased if the distribution of the values $P(x)$ is close to the uniform distribution on $\F$; otherwise we say that $P$ is biased. It is an important direction of research in higher order Fourier analysis to understand the structure of biased polynomials.

Note that a generic degree $d$ polynomial should be unbiased. In fact, as we will see below, if a degree $d$ polynomial is biased, then it can be written as a function of not too many polynomials of degree at most $d-1$. Let us now make this discussion more precise.

\begin{definition}
	Let $\F$ be a finite field and let $\chi$ be a nontrivial character of $\F$. The \emph{bias} of a function $f:\F^n\rightarrow \F$ with respect to $\chi$ is defined to be $\bias_{\chi}(f)=\E_{x\in \F^{n}} \lbrack \chi(f(x)) \rbrack$. (Here and elsewhere in the paper $\E_{x\in G}h(x)$ denotes $\frac{1}{|G|}\sum_{x\in G}h(x)$.)
\end{definition}

\begin{remark}
	Most of the previous work is on the case $\F=\F_p$ with $p$ a prime, in which case the standard definition of bias is $\bias(f)=\E_{x\in \F^n} \omega^{f(x)}$ where $\omega=e^{\frac{2\pi i}{p}}$.
\end{remark}

\begin{definition}
	Let $P$ be a polynomial $\F^n\rightarrow \F$ of degree $d$. The \emph{rank} of $P$ (denoted $\rank(P)$) is defined to be the smallest integer $r$ such that there exist polynomials $Q_1,\dots,Q_r:\F^n\rightarrow \F$ of degree at most $d-1$ and a function $f:\F^r\rightarrow \F$ such that $P=f(Q_1,\dots,Q_r)$.
\end{definition}

As discussed above, it is known that if a polynomial has large bias, then it has low rank. The first result in this direction was proved by Green and Tao \cite{greentao} who showed that if $\F$ is a field of prime order and $P:\F^n\rightarrow \F$ is a polynomial of degree $d$ with $d<|\F|$ and $\bias(P)\geq \d>0$, then $\rank(P)\leq c(\F,\d,d)$. Kaufman and Lovett \cite{kaufmanlovett} proved that the condition $d<|\F|$ can be omitted. In both results, $c$ has Ackermann-type dependence on its parameters. Finally, Bhowmick and Lovett \cite{bhowmicklovett} proved that if $d<\text{char}(\F)$ and $\bias(P)\geq |\F|^{-s}$, then $\rank(P)\leq c'(d,s)$. The novelty of this result is that $c'$ does not depend on $\F$. However, it still has Ackermann-type dependence on $d$ and $s$.

One of our main results is the following theorem, which improves the result of Bhowmick and Lovett unless $\F$ is very large. In this result, and in the rest of the paper, $\tower_{8|\F|}(h,x)$ denotes a tower of $8|\F|$'s of height $h$ with an $x$ on top, that is, $\tower_{8|\F|}(0,x)=x$ and $\tower_{8|\F|}(h,x)=(8|\F|)^{\tower_{8|\F|}(h-1,x)}$.

\begin{theorem} \label{biaslowrank}
	Let $\F$ be a finite field and let $\chi$ be a nontrivial character of $\F$. Let $P$ be a polynomial $\F^n\rightarrow \F$ of degree $d<\text{char}(\F)$. Suppose that $\bias_{\chi}(P)\geq \e>0$. Then $$\rank(P)\leq 2^{d}\tower_{8|\F|}((d+3)^{d+3},(1/\e)^{2^d})+1$$
\end{theorem}

Recall that if $G$ is an Abelian group and $d$ is a positive integer, then the Gowers $U^d$ norm (which is only a seminorm for $d=1$) of $f:G\rightarrow \C$ is defined to be
$$\|f\|_{U^{d}}=\big|\E_{x,y_1,\dots,y_d\in G} \prod_{S\s \lbrack d\rbrack} \cC^{d-|S|} f(x+\sum_{i\in S}y_i)\big|^{1/2^d},$$
where $\cC$ is the conjugation operator. It is a major area of research to understand the structure of functions $f$ whose $U^d$ norm is large. Our next theorem is a result in this direction.

\begin{theorem} \label{largenormlowrank}
	Let $\F$ be a finite field and let $\chi$ be a nontrivial character of $\F$. Let $P$ be a polynomial $\F^n\rightarrow \F$ of degree $d<\text{char}(\F)$. Let $f(x)=\chi(P(x))$ and assume that $\|f\|_{U^{d}}\geq c>0$. Then $$\rank(P)\leq 2^{d}\tower_{8|\F|}((d+3)^{d+3},(1/c)^{2^d})+1$$
\end{theorem}

Our result implies a similar improvement to the bounds for the quantitative inverse theorem for Gowers norms for polynomial phase functions of degree $d$.

\begin{theorem} \label{inverse}
	Let $\F$ be a field of prime order and let $P$ be a polynomial $\F^n\rightarrow \F$ of degree $d<\text{char}(\F)$. Let $f(x)=\omega^{P(x)}$ where $\omega=e^{\frac{2\pi i}{|\F|}}$ and assume that $\|f\|_{U^d}\geq c>0$. Then there exists a polynomial $Q:\F^n\rightarrow \F$ of degree at most $d-1$ such that $$|\E_{x\in \F^n} \omega^{P(x)}\overline{\omega^{Q(x)}}|\geq |\F|^{-2^{d}\tower_{8|\F|}((d+3)^{d+3},(1/c)^{2^d})-1}$$
\end{theorem}

Theorems \ref{biaslowrank} and \ref{inverse} easily follow from Theorem \ref{largenormlowrank}.

\smallskip

\emph{Proof of Theorem \ref{biaslowrank}.} Note that when $f(x)=\chi(P(x))$, then $\|f\|^2_{U^1}=|\E_{x,y\in \F^n}\overline{f(x)}f(x+y)|=|\E_{x\in \F^n}f(x)|^2$, so $\|f\|_{U^1}=|E_{x\in \F^n}f(x)|=|\bias_{\chi}(P)|$. However, $\|f\|_{U^k}$ is increasing in $k$ (see eg. Claim 6.2.2 in \cite{fouriersurvey}), therefore $\|f\|_{U^d}\geq |\bias_{\chi}(P)|\geq \e$. The result is now immediate from Theorem \ref{largenormlowrank}.

\smallskip

\emph{Proof of Theorem \ref{inverse}.} By Theorem \ref{largenormlowrank}, there exists a set of $r\leq 2^{d}\tower_{8|\F|}((d+3)^{d+3},(1/c)^{2^d})+1$  polynomials $Q_1,\dots,Q_r$ such that $P(x)$ is a function of $Q_1(x),\dots,Q_r(x)$. \newline Then $\omega^{P(x)}=g(Q_1(x),\dots,Q_r(x))$ for some function $g:\F^r\rightarrow \C$. Let $G=\F^r$. Note that $|g(y)|=1$ for all $y\in G$, therefore $|\hat{g}(\chi)|\leq 1$ for every character $\chi\in \hat{G}$. Now $\omega^{P(x)}=\sum_{\chi\in \hat{G}} \hat{g}(\chi)\,\chi((Q_1(x),\dots,Q_r(x))$, so $$1=\E_{x\in \F^n} |\omega^{P(x)}|^2=\sum_{\chi \in \hat{G}} \overline{\hat{g}(\chi)}\bigg( \E_{x\in \F^n} \omega^{P(x)}\overline{\chi(Q_1(x),\dots,Q_r(x))} \bigg).$$ Thus, there exists some $\chi\in \hat{G}$ with $|\E_{x\in \F^n} \omega^{P(x)}\overline{\chi(Q_1(x),\dots,Q_r(x))}|\geq 1/|G|=1/|\F|^r$. But $\chi$ is of the form $\chi(y_1,\dots,y_r)=\omega^{\sum_{i\leq r} \a_i y_i}$ for some $\a_i\in \F$. Then $\chi(Q_1(x),\dots,Q_r(x))=\omega^{Q_{\a}(x)}$, where $Q_{\a}$ is the degree $d-1$ polynomial $Q_{\a}(x)=\sum_{i\leq r} \a_iQ_i(x)$. So $Q=Q_{\a}$ is a suitable choice.

\subsection{Analytic rank and partition rank of tensors}

Related to the bias and rank of polynomials are the notions of \emph{analytic rank} and \emph{partition rank} of tensors. Recall that if $\F$ is a field and $V_1,\dots,V_d$ are finite dimensional vector spaces over $\F$, then an order $d$ tensor is a multilinear map $T:V_1\times \dots \times V_d\rightarrow \F$. Each $V_k$ can be identified with $\F^{n_k}$ for some $n_k$, and then there exist $t_{i_1,\dots,i_d}\in \F$ for all $i_1\leq n_1,\dots,i_d\leq n_d$ such that $T(v^1,\dots,v^d)=\sum_{i_1\leq n_1,\dots,i_d\leq n_d} t_{i_1,\dots,i_d} v^1_{i_1}\dots v^d_{i_d}$ for every $v^1\in \F^{n_1},\dots,v^d\in \F^{n_d}$ (where $v_k$ is the $k$th coordinate of the vector $v$). Indeed, $t_{i_1,\dots,i_d}$ is just $T(e^{i_1},\dots,e^{i_d})$, where $e^i$ is the $i$th standard basis vector.

The following notion was introduced by Gowers and Wolf \cite{gowerswolf}.

\begin{definition}
	Let $\F$ be a finite field, let $V_1,\dots,V_d$ be finite dimensional vector spaces over $\F$ and let $T:V_1\times \dots \times V_d\rightarrow \F$ be an order $d$ tensor. Then the \emph{analytic rank} of $T$ is defined to be $\arank(T)=-\log_{|\F|} \bias(T)$, where $\bias(T)=\E_{v^1\in V_1,\dots,v^d\in V_d} \lbrack \chi(T(v^1,\dots,v^d))\rbrack$ for any nontrivial character $\chi$ of $\F$.
\end{definition}

\begin{remark} \label{welldefined}
	This is well-defined. Indeed, if $\chi$ is a nontrivial character of $\F$, then
	\begin{align*}
		\E_{v^1\in V_1,\dots,v^d\in V_d} \lbrack \,\chi(T(v^1,\dots,v^d))\rbrack&=\E_{v^1\in V_1,\dots,v^{d-1}\in V_{d-1}}\lbrack \E_{v^d\in V_d} \,\chi(T(v^1,\dots,v^d))\rbrack \\
		&=\P_{v^1\in V_1,\dots,v^{d-1}\in V_{d-1}} \lbrack T(v^1,\dots,v^{d-1},x)\equiv 0\rbrack,
	\end{align*}
	where $T(v^1,\dots,v^{d-1},x)$ is viewed as a function in $x$. The second equality holds because \newline $\E_{v^d\in V_d} \,\chi(T(v^1,\dots,v^d))=0$ unless $T(v^1,\dots,v^{d-1},x)\equiv 0$, in which case it is 1.
	
	Thus, $\E_{v^1\in V_1,\dots,v^d\in V_d} \lbrack \,\chi(T(v^1,\dots,v^d))\rbrack$ does not depend on $\chi$, and is always positive. Moreover, it is at most 1, therefore the analytic rank is always nonnegative.
\end{remark}

\smallskip

A different notion of rank was defined by Naslund \cite{naslund}.

\begin{definition}
	Let $T:V_1\times \dots \times V_d\rightarrow \F$ be a (non-zero) order $d$ tensor. We say that $T$ has \emph{partition rank} 1 if there is some $S\s \lbrack d\rbrack$ with $S\neq \emptyset,S\neq \lbrack d\rbrack$ such that $T(v^1,\dots,v^d)=T_1(v^{i}:i\in S)T_2(v^{i}:i\not \in S)$ where $T_1:\prod_{i\in S}V_i\rightarrow \F,T_2:\prod_{i\not\in S}V_i\rightarrow \F$ are tensors. In general, the partition rank of $T$ is the smallest $r$ such that $T$ can be written as the sum of $r$ tensors of partition rank 1. This number is denoted $\prank(T)$.
\end{definition}

Lovett \cite{lovett} has proved that $\arank(T)\leq \prank(T)$. In the other direction, it follows from the work of Bhowmick and Lovett \cite{bhowmicklovett} that if an order $d$ tensor $T$ has $\arank(T)\leq r$, then $\prank(T)\leq f(r,d)$ for some function $f$. Note that $f$ does not depend on $|\F|$ or the dimension of the vector spaces $V_k$. However, $f$ has an Ackermann-type dependence on $d$ and $r$. We prove a different bound under the same assumptions, which is stronger unless $|\F|$ is very large.

\begin{theorem} \label{analyticpartition}
	Let $T:V_1\times \dots \times V_d\rightarrow \F$ be an order $d$ tensor with $\arank(T)\leq r$. Then
	$$\prank(T)\leq 2^{d-1}\tower_{8|\F|}((d+3)^{d+3}+1,r).$$
\end{theorem}

It is not hard to see that Theorem \ref{analyticpartition} implies Theorem \ref{largenormlowrank}. Indeed, let $P$ be a polynomial $\F^n\rightarrow \F$ of degree $d<\text{char}(\F)$, let $f(x)=\chi(P(x))$ and assume that $\|f\|_{U^d}\geq c>0$. Define $T:(\F^n)^d \rightarrow \F$ by $T(y_1,\dots,y_d)=\sum_{S\s \lbrack d\rbrack} (-1)^{d-|S|}P(\sum_{i\in S} y_i)$. By Lemma 2.4 from \cite{gowerswolf}, $T$ is a tensor of order $d$. Moreover, by the same lemma, we have $T(y_1,\dots,y_d)=\sum_{S\s \lbrack d\rbrack} (-1)^{d-|S|}P(x+\sum_{i\in S} y_i)$ for any $x\in \F^n$. Thus, $$\bias(T)=\E_{y_1,\dots,y_d\in \F^n} \, \chi(T(y_1,\dots,y_d))=\E_{y_1,\dots,y_d\in \F^n} \prod_{S\s \lbrack d\rbrack} \cC^{d-|S|}f(x+\sum_{i\in S} y_i)$$ for any $x\in \F^n$. By averaging over all $x\in \F^n$, it follows that $\bias(T)=\|f\|^{2^d}_{U^d}\geq c^{2^d}$. Thus, $\arank(T)\leq 2^d\log_{|\F|}(1/c)$. Therefore, by Theorem \ref{analyticpartition}, we get
\begin{equation} \label{eqnprank}
	\prank(T)\leq 2^{d-1}\tower_{8|\F|}\big((d+3)^{d+3}+1,2^d\log_{|\F|}(1/c)\big)= 2^{d-1}\tower_{8|\F|}\big((d+3)^{d+3},(1/c)^{2^d}\big).
\end{equation}
Note that $T(y_1,\dots,y_d)=D_{y_1}\dots D_{y_d}P(x)$ where $D_yg(x)=g(x+y)-g(x)$. Thus, by Taylor's approximation theorem, since $d<\text{char}(\F)$, we get $P(x)=\frac{1}{d!}D_x \dots D_x P(0)+W(x)=\frac{1}{d!}T(x,\dots,x)+W(x)$ for some polynomial $W$ of degree at most $d-1$.



By equation (\ref{eqnprank}), $T$ can be written as a sum of at most $2^{d-1}\tower_{8|\F|}((d+3)^{d+3},(1/c)^{2^d})$ tensors of partition rank 1. Hence, $P_1$ can be written as a sum of at most $2^{d-1}\tower_{8|\F|}((d+3)^{d+3},(1/c)^{2^d})$ expressions of the form $QR$ where $Q,R$ are polynomials of degree at most $d-1$ each. Thus, $P-W$ has rank at most $2\cdot2^{d-1}\tower_{8|\F|}((d+3)^{d+3},(1/c)^{2^d})$, and therefore $P$ has rank at most $$2\cdot2^{d-1}\tower_{8|\F|}((d+3)^{d+3},(1/c)^{2^d})+1.$$

\section{The proof of Theorem \ref{analyticpartition}}

\subsection{Notation}

In the rest of the paper, we identify $V_i$ with $\F^{n_i}$. Thus, the set of all tensors $V_1\times \dots \times V_d\rightarrow \F$ is the tensor product $\F^{n_1}\otimes \dots \otimes \F^{n_d}$, which will be denoted by $\cG$ throughout this section. Also, $\cB$ will always stand for the multiset $\{u_1\otimes\dots\otimes u_d:u_i\in \F^{n_i} \text{ for all } i\}$. Note that $\cG=\F^{n_1}\otimes\dots\otimes\F^{n_d}$ can be viewed as the set of $d$-dimensional $(n_1,\dots,n_d)$-arrays over $\F$ which in turn can be viewed as $\F^{n_1n_2\dots n_d}$, equipped with the entry-wise dot product.

For $I\s \lbrack d\rbrack$, we write $\F^{I}$ for $\bigotimes_{i\in I} \F^{n_i}$ so that we naturally have $\cG=\F^{I}\otimes \F^{I^c}$, where $I^c$ always denotes $\lbrack d\rbrack \setminus I$.

If $r\in \F^{\lbrack d\rbrack}=\cG$ and $s\in \F^{\lbrack k\rbrack}$ (for some $k\leq d$), then we define $rs$ to be the tensor in $\F^{\lbrack k+1,d\rbrack}$ with coordinates $(rs)_{i_{k+1},\dots,i_d}=\sum_{i_1\leq n_1,\dots,i_k\leq n_k} r_{i_1,\dots,i_d}s_{i_1,\dots,i_k}$. If $k=d$, then $rs$ is the same as the entry-wise dot product $r.s$. Also, note that viewing $r$ as a $d$-multilinear map $R:\F^{n_1}\times \dots \times \F^{n_d}\rightarrow \F$, we have $R(v^1,\dots,v^d)=\sum_{i_1\leq n_i,\dots,i_d\leq n_d} r_{i_1,\dots,i_d}v^1_{i_1}\dots v^d_{i_d}=r(v^1\otimes \dots \otimes v^d)$.

Finally, we use a non-standard notation and write $kB$ to mean the set of elements of $\cG$ which can be written as a sum of \textit{at most} $k$ elements of $B$, where $B$ is some fixed (multi)subset of $\cG$, and similarly, we write $kB-lB$ for the set of elements that can be obtained by adding at most $k$ members and subtracting at most $l$ members of $B$.

\subsection{The main lemma and some consequences}

Theorem \ref{analyticpartition} will follow easily from the next lemma, which is the main technical result of this paper. See \cite{gowersjanzer} for another application of this lemma.

\begin{lemma} \label{mainlemma}
	Let $f_1(d)=2^{3^{d+3}}$, $f_2(d)=2^{-3^{d+3}}$ and $f_3(d,\d)=\tower_{8|\F|}((d+4)^{d+4},1/\d)$. If $\cB'\subset \cB$ is a multiset such that $|\cB'|\geq \d |\cB|$, then there exists a multiset $Q$ whose elements are chosen from $f_1(d)\cB'-f_1(d)\cB'$ (but with arbitrary multiplicity) with the following property. The set of arrays $r\in \cG$ with $r.q=0$ for at least $(1-f_2(d))|Q|$ choices $q\in Q$ is contained in $\sum_{I\subset \lbrack d\rbrack, I\neq \emptyset} V_I\otimes \F^{I^c}$ for subspaces $V_I\subset \F^{I}$ of dimension at most $f_3(d,\d)$.
\end{lemma}

The proof of this lemma goes by induction on $d$. In what follows, we shall prove results conditional on the assumption that Lemma \ref{mainlemma} has been verified for all $d'<d$. Eventually, we will use these results to prove the induction step.

\begin{definition}
	Let $k$ be a positive integer. We say that $r\in \cG$ is \emph{$k$-degenerate} if for every $I\s \lbrack d\rbrack,I\neq \emptyset,I\neq \lbrack d\rbrack$, there exists a subspace $H_I\s \F^{I}$ of dimension at most $k$ such that $r\in \sum_{I\subset \lbrack d-1 \rbrack,I\neq \emptyset}H_I\otimes H_{I^c}$.
\end{definition}

If $r\in H_I\otimes \F^{I^c}$ with $\dim(H_I)\leq k$, then $r\in H_I\otimes H_{I^c}$ for some $H_{I^c}\s \F^{I^c}$ of dimension at most $k$. (This follows by writing $r$ as $\sum_{j\leq m} s_j\otimes t_j$ with $\{s_j\}$ a basis for $H_I$ and letting $H_{I^c}$ be the span of all the $t_j$.) Thus, $r$ is $k$-degenerate if and only if $r\in \sum_{I\s \lbrack d-1 \rbrack, I\neq \emptyset} H_I\otimes \F^{I^c}$ for some $H_I\s \F^{I}$ of dimension at most $k$, or equivalently, if and only if $r\in \sum_{I\s \lbrack d-1 \rbrack, I\neq \emptyset} \F^{I}\otimes H_{I^c}$ for some $H_{I^c}\s \F^{I^c}$ of dimension at most $k$. Moreover, note that if $r$ is $k$-degenerate, then $\prank(r)\leq 2^{d-1}k$. This is because if $I\neq \emptyset,I\s \lbrack d-1 \rbrack$ and $w\in H_I\otimes H_{I^c}$ for subspaces $H_I\s \F^{I}$ and $H_{I^c}\s \F^{I^c}$ of dimension at most $k$, then $w=\sum_{i\leq k} s_i\otimes t_i$ for some $s_i\in H_{I}$, $t_i\in H_{I^c}$. But clearly, $s_i\otimes t_i$ has partition rank 1.

\begin{lemma} \label{biastorank}
	Suppose that Lemma \ref{mainlemma} has been proved for $d'=d-1$. Let $r\in \cG$ be such that $r(v_1\otimes \dots \otimes v_{d-1})=0\in \F^{n_d}$ for at least $\d|\F|^{n_1\dots n_{d-1}}$ choices $v_1\in \F^{n_1},\dots,v_{d-1}\in \F^{n_{d-1}}$. Then $r$ is $f$-degenerate for $f=\tower_{8|\F|}((d+3)^{d+3},1/\d)$.
\end{lemma}

\begin{proof}
	Write $r=\sum_i s_i\otimes t_i$ where $s_i\in \F^{\lbrack d-1\rbrack}$ and $\{t_i\}_i$ is a basis for $\F^{n_d}$. Let $\cD$ be the multiset $\{u_1\otimes \dots \otimes u_{d-1}:u_1\in \F^{n_1},\dots,u_{d-1}\in\F^{n_{d-1}}\}$ and $\cD'=\{w\in \cD: rw=0\}$. Since $|\cD'|\geq \d|\cD|$, by Lemma \ref{mainlemma} there is a multiset $Q$ with elements from $2^{3^{d+2}}\cD'$ such that the set of arrays $r'\in \F^{\lbrack d-1\rbrack}$ with $r'.q=0$ for all choices $q\in Q$ is contained in some $\sum_{I\s \lbrack d-1\rbrack,I\neq \emptyset} V_I\otimes \F^{\lbrack d-1\rbrack \setminus I}$, where $\dim(V_I)\leq \tower_{8|\F|}((d+3)^{d+3},1/\d)$. Note that for every $i$ we have $s_i.w=0$ for all $w\in \cD'$ and so also $s_i.q=0$ for all $q\in Q$. Thus, $r\in \sum_{I\s \lbrack d-1\rbrack,I\neq \emptyset} V_I\otimes \F^{I^c}$.
\end{proof}

Now we are in a position to prove Theorem \ref{analyticpartition} conditional on Lemma \ref{mainlemma}.

\begin{proof}[Proof of Theorem \ref{analyticpartition}]
	Let $T:\F^{n_1}\times \dots \times \F^{n_d}\rightarrow \F$ be an order $d$ tensor with $\arank(T)\leq r$. By Remark \ref{welldefined}, we have $\P_{v_1\in \F^{n_1},\dots,\F^{n_{d-1}}\in V_{d-1}}\lbrack T(v_1,\dots,v_{d-1},x)\equiv 0\rbrack\geq |\F|^{-r}$. Writing $t$ for the element in $\cG$ corresponding to $T$, we get that $t(v_1\otimes \dots \otimes v_{d-1}\otimes x)\equiv 0$ as a function of $x$ for at least $\d|\F|^{n_1\dots n_d}$ choices $v_1\in \F^{n_1},\dots,v_{d-1}\in \F^{n_{d-1}}$, where $\d=|\F|^{-r}$. But $t(v_1\otimes \dots \otimes v_{d-1}\otimes x)=\big(t(v_1\otimes \dots \otimes v_{d-1})\big).x$, so we have $t(v_1\otimes \dots \otimes v_{d-1})=0$ for all these choices of $v_i$. Thus, by Lemma \ref{biastorank}, $t$ is $f$-degenerate for $f=\tower_{8|\F|}((d+3)^{d+3},|\F|^r)$. Hence, $\prank(T)\leq 2^{d-1}\tower_{8|\F|}((d+3)^{d+3},|\F|^r)= 2^{d-1}\tower_{8|\F|}((d+3)^{d+3}+1,r)$.
\end{proof}

Let us continue the preparation for the induction step in Lemma \ref{mainlemma}.

\begin{lemma} \label{focusondeg}
	Suppose that Lemma \ref{mainlemma} has been proved for $d'=d-1$. Let $\cB'\s \cB$ be such that $|\cB'|\geq \d |\cB|$ for some $\d>0$. Then there exist some $Q\subset 2\cB'-2\cB'$ and a subspace $V_{\lbrack d\rbrack}\s \F^{\lbrack d\rbrack}$ of dimension at most $5|\F|^{4/\d^2}$ with the following property. Any array $r$ with $r.q=0$ for at least $\frac{3}{4}|Q|$ choices $q\in Q$ can be written as $r=x+y$ where $x\in V_{\lbrack d \rbrack}$ and $y$ is $f$-degenerate for $f=\tower_{8|\F|}((d+3)^{d+3}+3,1/\d)$.
\end{lemma}

\begin{proof}	
	Let $\cD$ be the multiset $\{u_1\otimes \dots \otimes u_{d-1}:u_1\in \F^{n_1},\dots,u_{d-1}\in\F^{n_{d-1}}\}$ and $\cD'=\{t\in \cD: t\otimes u\in \cB' \text{ for at least } \frac{\d}{2}|\F|^{n_d} \text{ choices } u\in \F^{n_d}\}$. Clearly, we have $|\cD'|\geq \frac{\d}{2}|\cD|$. Moreover, by Bogolyubov's lemma (see, eg. Proposition 4.39 in \cite{taovu}), for every $t\in \cD'$, there exists a subspace $U_t\s \F^{n_d}$ of codimension at most $\frac{1}{(\d/2)^2}$ such that $t\otimes U_t\s 2\cB'-2\cB'$ for every $t\in \cD'$. After passing to suitable subspaces, we may assume that all $U_t$ have the same codimension $k\leq 4/\d^2$. Now let $Q=\cup_{t\in \cD'}(t\otimes U_t)$.
	
	Let $r_1,\dots,r_m\in \cG$ have the properties that
	
	\begin{enumerate} [label=(\roman*)]
		\item for all $i$, we have $r_i.q=0$ for at least $\frac{3}{4}|Q|$ choices $q\in Q$ and
		
		\item for all $i\neq j$, $r_i-r_j$ is not $f$-degenerate.
	\end{enumerate}

	Note that $r_i.(t\otimes s)=(r_it).s$ for every $s\in U_t$. If $r_it\not \in U_t^{\perp}$, then $(r_it).s=0$ holds for only a proportion $1/|\F|\leq 1/2$ of all $s\in U_t$. Thus, by (i) it follows that for all $i$, we have $r_it\in U_t^{\perp}$ for at least $\frac{1}{2}|\cD'|$ choices $t\in \cD'$.
	
	Suppose that $m=5|\F|^{k}$. By averaging, for at least $\frac{\cD'}{4}$ choices $t\in \cD'$ there are at least $m/4$ choices $i\leq m$ with $r_it\in U_t^{\perp}$. For every such $t$ there exist $i\neq j$ with $r_it=r_jt$ since $|U_t^{\perp}|\leq |\F|^k$. Thus, there exist $i\neq j$ such that $r_it=r_jt$ holds for at least $\frac{|\cD'|}{4m^2}$ choices $t\in \cD'$. Take $\tilde{r}=r_i-r_j$. Then $\tilde{r}t=0$ for at least $\frac{\d}{8m^2}|\cD|$ choices $t\in \cD$. By Lemma \ref{biastorank}, $\tilde{r}$ is $g$-degenerate for $g=\tower_{8|\F|}((d+3)^{d+3},\frac{8m^2}{\d})$.  But $\frac{8m^2}{\d}\leq\frac{200}{\d}|\F|^{8/\d^2}\leq (8|\F|)^{32/\d^2}$. Moreover, $(8|\F|)^{(8|\F|)^{1/\d}}\geq 16^{16^{1/\d}}\geq \frac{32}{\d^2}$. Thus, $\tilde{r}$ is in fact $f$-degenerate, contradicting (ii).
	
	Thus, if $r_1,\dots,r_m$ is a maximal set with properties (i) and (ii), then $m< 5|\F|^k$. Hence we may take $V_{\lbrack d \rbrack}$ to be the span of $r_1,\dots,r_m$ and this satisfies the conclusion of the lemma.
\end{proof}

\begin{remark}
	In the proof above and later in the paper we are using the bound $1/\d^2$ on the codimension of the subspace obtained in Bogolyubov's lemma (where $\d$ is the density of our initial set). This is not the best known bound but this choice is simple and makes no difference in the final bound in Lemma \ref{mainlemma}. Later (see Remark \ref{expensive}) we will highlight the most expensive step of the argument.
\end{remark}

\subsection{Construction of some auxiliary sets}

The next definition describes a type of set that will be useful for us when constructing $Q$ in Lemma \ref{mainlemma}. Its key properties are described in this subsection.

\begin{definition} \label{ksystem}
	Suppose that we have a collection of vector spaces as follows. The first one is $U\s \F^{n_1}$, of codimension at most $l$. Then, for every $u_1\in U$, there is some $U_{u_1}\s \F^{n_2}$. In general, for every $2\leq k\leq d$ and every $u_1\in U,u_2\in U_{u_1},\dots,u_{k-1}\in U_{u_1,\dots,u_{k-2}}$, there is a subspace $U_{u_1,\dots,u_{k-1}}\s \F^{n_k}$. Assume, in addition, that the codimension of $U_{u_1,\dots,u_{k-1}}$ in $\F^{n_k}$ is at most $l$ for every $u_1\in U,\dots,u_{k-1}\in U_{u_1,\dots,u_{k-2}}$. Then the multiset $Q=\{u_1\otimes \dots \otimes u_d: u_1\in U,\dots,u_d\in U_{u_1,\dots,u_{d-1}}\}$ is called an $l$-system.
	
\end{definition}

\begin{lemma} \label{intersect}
	Let $Q$ be an $l$-system and let $Q'$ be a $l'$-system. Then $Q\cap Q'$ contains an $(l+l')$-system.
\end{lemma}

\begin{proof}
	Let $Q$ have spaces as in Definition \ref{ksystem} and let $Q'$ have spaces $U'_{u'_1,\dots,u'_{k-1}}$.
	We define an $(l+l')$-system $P$ contained in $Q\cap Q'$ as follows. Let $V=U\cap U'$. Suppose we have defined $V_{v_1,\dots,v_{j-1}}$ for all $j\leq k$. Let $v_1\in V,v_2\in V_{v_1},\dots,v_{k-1}\in V_{v_1,\dots,v_{k-2}}$.  We let $V_{v_1\dots,v_{k-1}}=U_{v_1\dots,v_{k-1}}\cap U'_{v_1\dots,v_{k-1}}$. This is well-defined and has codimension at most $l+l'$ in $\F^{n_k}$. Let $P$ be the $(l+l')$-system with spaces $V_{v_1,\dots,v_{k-1}}$.
\end{proof}

\begin{lemma} \label{findsystem}
	Let $\cB'\subset \cB$ be a multiset such that $|\cB'|\geq \delta |\cB|$. Then there exists an $f_1$-system whose elements are chosen from $f_2\cB'-f_2\cB'$ with $f_1=\frac{16^{d}}{\d^2}$ and $f_2=4^d$.
\end{lemma}

\begin{proof}
	The proof is by induction on $d$. The case $d=1$ is an easy application of Bogolyubov's lemma. Suppose that the lemma has been proved for all $d'<d$ and let $\cB'\subset \cB$ be a multiset such that $|\cB'|\geq \delta |\cB|$. Let $\cD$ be the multiset $\{v_2\otimes\dots\otimes v_d: v_2\in \F^{n_2},\dots,v_d\in \F^{n_d}\}$. For each $u\in \F^{n_1}$, let $\cB_u'=\{s\in \cD: u\otimes s\in \cB'\}$ and let $T=\{u\in \F^{n_1}: |\cB_u'|\geq\frac{\delta}{2}|\cD|\}$. By averaging, we have that $|T|\geq\frac{\delta}{2}|\F|^{n_1}$. Now by the induction hypothesis, for every $t\in T$, there exists a $g_1$-system in $\F^{n_2}\otimes \dots \otimes \F^{n_d}$ (whose definition is analogous to the definition of a system in $\F^{n_1}\otimes \dots \F^{n_d}$), called $P_t$, contained in $g_2\cB'_t-g_2\cB'_t$ where $g_1=\frac{16^{d-1}}{(\d/2)^2}$ and $g_2=4^{d-1}$. By Bogolyubov's lemma, $2T-2T$ contains a subspace $U\s \F^{n_1}$ of codimension at most $\frac{1}{(\d/2)^2}$. For each $u\in U$, write $u=t_1+t_2-t_3-t_4$ arbitrarily with $t_i\in T$, and let $Q_u=P_{t_1}\cap P_{t_2}\cap P_{t_3}\cap P_{t_4}$, which is a $g_3$-system with $g_3=4g_1=\frac{16^d}{\d^2}$, by Lemma \ref{intersect}. Thus, $Q=\bigcup_{u\in U} (u\otimes Q_u)$ is indeed an $f_1$-system. Moreover, for any $u\in U,s\in Q_u$, we have $u\otimes s=t_1\otimes s+t_2\otimes s-t_3\otimes s-t_4\otimes s$ for some $t_i\in T$ and $s\in \bigcap_{i\leq 4} P_{t_i}$. Then $t_i\otimes s\in g_2\cB'-g_2\cB'$, therefore $u\otimes s\in 4g_2\cB'-4g_2\cB'$, so the elements of $Q$ are indeed chosen from $f_2\cB'-f_2\cB'$.
\end{proof}

\begin{lemma} \label{systemspace}
	Let $Q$ be a $k$-system and for every non-empty $I\s \lbrack d\rbrack$, let $L_I\s \F^{I}$ be a subspace of codimension at most $l$. Let $T=\bigcap_I(L_I\otimes \F^{I^c})$. Then $Q\cap T$ contains an $f$-system for $f=k+2^dl$.
\end{lemma}

\begin{proof}
	Let the spaces of $Q$ be $U_{u_1,\dots,u_{j-1}}$. It suffices to prove that for every $1\leq j\leq d$, and every $u_1\in U,\dots,u_{j-1}\in U_{u_1,\dots,u_{j-2}}$, the codimension of $(u_1\otimes \dots\otimes u_{j-1}\otimes U_{u_1,\dots,u_{j-1}})\cap \bigcap_{I\subset \lbrack j\rbrack, j\in I}(L_I\otimes \F^{\lbrack j\rbrack \setminus I})$ in $u_1\otimes \dots\otimes u_{j-1}\otimes U_{u_1,\dots,u_{j-1}}$ is at most $2^dl$. Thus, it suffices to prove that for every $I\s \lbrack j\rbrack$ with $j\in I$, the codimension of $(u_1\otimes \dots\otimes u_{j-1}\otimes U_{u_1,\dots,u_{j-1}})\cap(L_I\otimes \F^{\lbrack j\rbrack \setminus I})$ in $u_1\otimes \dots\otimes u_{j-1}\otimes U_{u_1,\dots,u_{j-1}}$ is at most $l$. But this is equivalent to the statement that $\big((\bigotimes_{i\in I\setminus \{j\}}u_i)\otimes U_{u_1,\dots,u_{j-1}}\big)\cap L_I$ has codimension at most $l$ in $(\bigotimes_{i\in I\setminus \{j\}}u_i)\otimes U_{u_1,\dots,u_{j-1}}$, which clearly holds.
\end{proof}

\subsection{Finishing the proof of Lemma \ref{mainlemma}}

\begin{definition}
	Let $k$ be a positive integer and let $\e>0$. Let $Q$ be a multiset with elements chosen from $\cG$ (with arbitrary multiplicity). We say that $Q$ is $(k,\a)$-forcing if the set of all arrays $r\in \cG$ with $r.q=0$ for at least $\a|Q|$ choices $q\in Q$ is contained in a set of the from $\sum_{I\subset \lbrack d\rbrack,I\neq \emptyset} V_I\otimes \F^{I^c}$ for some $V_I\s \F^{I}$ of dimension at most $k$.
\end{definition}

We now turn to the main part of the proof of Lemma \ref{mainlemma}. For each $I\s \lbrack d-1 \rbrack$ we will construct a corresponding $Q_I$ as defined in the next result, and (roughly) we will take $Q=\bigcup_I Q_I$.

\begin{lemma} \label{claim}
	Suppose that Lemma \ref{mainlemma} has been proved for every $d'<d$. Let $\cB'\s \cB$ have $|\cB'|\geq \d|\cB|$ for some $\d>0$. Let $k\geq \tower_{8|\F|}((d+3)^{d+3},1/\d)$ arbitrary, let $I\s \lbrack d-1\rbrack, I\neq \emptyset$, and let $W_J\s \F^{J}$ be subspaces of dimension at most $k$ for every $J\s I,J\neq I,J\neq \emptyset$. Then there exist a multiset $Q'$, and a multiset $Q_s$ for each $s\in Q'$ with the following properties.
	
	\begin{enumerate}[label=(\arabic*)]
		\item The elements of $Q'$ are chosen from $\bigcap_{J\subset I,J\neq I,J\neq \emptyset}(W_J^{\perp}\otimes \F^{I \setminus J})$
		
		\item $Q'$ is $(f_1,1-f_2)$-forcing with $f_1=\tower_{8|\F|}((d+3)^{|I|+4}+2,k)$, $f_2=2^{-3^{d+2}}$
		
		\item For each $s\in Q'$, the elements of $Q_s$ are chosen from $\F^{I^c}$
		
		\item For each $s\in Q'$, $Q_s$ is $(f_3,1-f_4)$-forcing with $f_3=\tower_{8|\F|}(4,k)$, $f_4=2^{-3^{d+2}}$
		
		\item $\max_{s\in Q'}|Q_s|\leq 2\min_{s\in Q'}|Q_s|$
		
		\item The elements of the multiset $Q_I:=\{s\otimes t: s\in Q',t\in Q_s\}=\bigcup_{s\in Q'}(s\otimes Q_s)$ are chosen from $f_5\cB'-f_5\cB'$ with $f_5=2^{3^{d+3}}$.
	\end{enumerate}
\end{lemma}

\begin{proof}
	By symmetry, we may assume that $I=\lbrack a\rbrack$ for some $1\leq a \leq d-1$. Let $\cC$ be the multiset $\{u_1\otimes \dots \otimes u_a: u_i\in \F^{n_i}\}$ and let $\cD$ be the multiset $\{u_{a+1}\otimes \dots \otimes u_d: u_i\in \F^{n_i}\}$. For each $s\in \cC$, let $\cD_s=\{t\in \cD: s\otimes t\in \cB'\}$. Also, let $\cC'=\{s\in \cC: |\cD_s|\geq \frac{\d}{2}|\cD|\}$. Clearly, $|\cC'|\geq \frac{\d}{2}|\cC|$. By Lemma \ref{findsystem}, there exists a $g_1$-system $R$ (with respect to $\F^{I}$) with elements chosen from $g_2\cC'-g_2\cC'$ with $g_1=\frac{4\cdot 16^d}{\d^2}$ and $g_2=4^d$. By Lemma \ref{systemspace}, $R\cap \bigcap_{J\subset I,J\neq I,J\neq \emptyset} (W_J^{\perp}\otimes \F^{I\setminus J})$ contains a $g_3$-system $T'$ for $g_3=\frac{4\cdot 16^d}{\d^2}+2^dk$. Now $|T'|\geq |\F|^{-dg_3}|\cC|$. By Lemma \ref{mainlemma} (applied to $a$ in place of $d$), it follows that there exists a multiset $Q'$ whose elements are chosen from $g_4T'-g_4T'$ and which is $(g_5,1-g_6)$-forcing for $g_4=2^{3^{a+3}}\leq 2^{3^{d+2}}$, $g_5=\tower_{8|\F|}((a+4)^{a+4},|\F|^{dg_3})\leq \tower_{8|\F|}((d+3)^{a+4},|\F|^{dg_3})$ and $g_6=2^{-3^{a+3}}\geq 2^{-3^{d+2}}$. But $g_3\leq 2\cdot 2^dk$, so $|\F|^{dg_3}\leq \tower_{8|\F|}(2,k)$, therefore $Q'$ satisfies (1) and (2) in the statement of this lemma.
	
	By Lemma \ref{findsystem}, for each $s\in \cC'$ there exists a $g_7$-system $R_s$ (with respect to $\F^{I^c}$) contained in $g_8\cD_s-g_8\cD_s$, where $g_7=\frac{4\cdot 16^d}{\d^2}$ and $g_8=4^d$. For every $s\in Q'$, choose $s_1,\dots,s_{l+l'}\in \cC'$ with $l,l'\leq 2^{3^{d+3}}$ such that $s=s_1+\dots+s_l-s_{l+1}-\dots-s_{l+l'}$ (this is possible, since the elements of $Q'$ are chosen from $2g_2g_4\cC'-2g_2g_4\cC'$ and $2g_2g_4\leq 2^{3^{d+3}}$), and let $P_s=\bigcap_{i\leq l+l'}R_s$. By Lemma \ref{intersect}, $P_s$ contains a $g_9$-system with $g_9=2\cdot 2^{3^{d+3}}\cdot \frac{4\cdot 16^d}{\d^2}$, therefore $|P_s|\geq g_{10}|\cD|$ for $g_{10}=|\F|^{-dg_9}$. By Lemma \ref{mainlemma} (applied to $d-a$ in place of $d$), for every $s\in Q'$ there exists a multiset $Q_s$ with elements chosen from $g_{11}P_s-g_{11}P_s$ which is $(g_{12},1-g_{13})$-forcing for $g_{11}=2^{3^{d-a+3}}\leq 2^{3^{d+2}}$, $g_{12}=\tower_{8|\F|}((d-a+4)^{d-a+4},|\F|^{dg_9})\leq \tower_{8|\F|}((d+3)^{d+3},|\F|^{dg_9})\leq \tower_{8|\F|}((d+3)^{d+3}+4,1/\d)\leq \tower_{8|\F|}(4,k)$ and $g_{13}=2^{-3^{d-a+3}}\geq 2^{-3^{d+2}}$. Notice that if we repeat every element of $Q_s$ the same number of times, then the multiset obtained is still $(g_{12},1-g_{13})$-forcing, so we may assume that $\max_{s\in Q'}|Q_s|\leq 2\min_{s\in Q'}|Q_s|$. Define $Q_I=\{s\otimes t:s\in Q',t\in Q_s\}=\bigcup_{s\in Q'}(s\otimes Q_s)$. Note that as $R_s\subset g_8\cD_s-g_8\cD_s$ for all $s\in \cC'$, we have $s\otimes R_s\subset g_8\cB'-g_8\cB'$ for all $s\in \cC'$. But the elements of $Q'$ are chosen from $2g_2g_4\cC'-2g_2g_4\cC'$, so $s\otimes P_s\s 4g_2g_4g_8\cB'-4g_2g_4g_8\cB'$ for all $s\in Q'$. Finally, the elements of $Q_s$ are chosen from $g_{11}P_s-g_{11}P_s$, so the elements of $s\otimes Q_s$ are chosen from $8g_2g_4g_8g_{11}\cB'-8g_2g_4g_8g_{11}\cB'$ for every $s\in Q'$. Since $8g_2g_4g_8g_{11}\leq 8\cdot (4^d)^2\cdot (2^{3^{d+2}})^2=2^{3+4d+2\cdot 3^{d+2}}\leq 2^{3^{d+3}}$, property (6) is satisfied.
\end{proof}

\begin{remark} \label{expensive}
	Forcing condition (1) in Lemma \ref{claim} is the most expensive step in the proof of Lemma \ref{mainlemma}. That condition is the reason why $f_1$ is so large in (2).
\end{remark}

\smallskip

We have already seen in Lemma \ref{focusondeg} that (for suitably chosen $Q$) the set of $r\in \cG$ that satisfy $r.q=0$ for most $q\in Q$ are of the form $r_1+r_2$ where $r_1$ lives in a fixed small subspace of $\cG$ and $r_2$ is $k$-degenerate for some small $k$. The next lemma allows us to turn the subspaces witnessing the $k$-degeneracy of $r_2$ into slightly larger subspaces which however do not depend on $r$. This is done one by one, in an order determined by which $\F^{I^c}$ the subspace lives in. The order in which these index sets are considered is not arbitrary: we define $\prec$ to be any total order on the set of non-empty subsets of $\lbrack d-1 \rbrack$ such that if $J\subsetneq I$ then $J\prec I$.

\begin{lemma} \label{keylemma}
	Let $k\geq \tower_{8|\F|}((d+3)^{d+3},1/\d)$ arbitrary. Let $I\s \lbrack d-1\rbrack, I\neq \emptyset$ and let $W_J\s \F^{J},W_{J^c}\s \F^{J^c}$ be subspaces of dimension at most $k$ for every $J\prec I$. Moreover, let $W_{\lbrack d\rbrack}\s \F^{\lbrack d\rbrack}$ have dimension at most $k$. Suppose that $Q',Q_s$ (and $Q_I$) have the six properties described in Lemma \ref{claim}. Then any array $r\in W_{\lbrack d\rbrack}+ \sum_{J\prec I} (W_J\otimes \F^{J^c}+\F^{J}\otimes W_{J^c})+\sum_{J\succeq I} \F^{J}\otimes H_{J^c}(r)$ with $\dim(H_{J^c}(r))\leq k$ and the property that $r.q=0$ for at least $(1-\frac{1}{4}(2^{-3^{d+2}})^2)|Q_I|$ choices $q\in Q_I$ is contained in $W_{\lbrack d \rbrack}+\sum_{J\preceq I} (U_J\otimes \F^{J^c}+\F^{J}\otimes U_{J^c})+\sum_{J\succ I} \F^{J}\otimes K_{J^c}(r)$ for some $U_J\s \F^{J},U_{J^c}\s \F^{J^c}$ not depending on $r$ and some $K_{J^c}(r)\s \F^{J^c}$ possibly depending on $r$, all of dimension at most $\tower_{8|\F|}((d+3)^{|I|+4}+3,k)$.
\end{lemma}

\begin{proof}
	By (4) in Lemma \ref{claim}, for every $s\in Q'$ there exist subspaces $V_J(s)\s \F^{J}$ for every $J\s I^c,J\neq \emptyset$, with dimension at most $g_1=\tower_{8|\F|}(4,k)$ such that the set of arrays $t\in \F^{I^c}$ with $t.q=0$ for at least $(1-g_2)|Q_s|$ choices $q\in Q_s$ is contained in $\sum_{J\s I^c,J\neq \emptyset} V_J(s)\otimes \F^{I^c\setminus J}$, where $g_2=2^{-3^{d+2}}$. If $r\in \cG$ has $r.q=0$ for at least $(1-\frac{1}{4}(2^{-3^{d+2}})^2)|Q_I|$ choices $q\in Q_I$, then by averaging and using (5) from Lemma \ref{claim}, for at least $(1-g_3)|Q'|$ choices $s\in Q'$ we have $r.(s\otimes t)=0$ for at least $(1-g_2)|Q_s|$ choices $t\in Q_s$, where $g_3=\frac{1}{2}2^{-3^{d+2}}$. Thus, (noting that $r.(s\otimes t)=(rs).t$), $rs\in \sum_{J\s I^c,J\neq \emptyset} V_J(s)\otimes \F^{I^c\setminus J}$ holds for at least $(1-g_3)|Q'|$ choices $s\in Q'$. Let $Q'(r)$ be the submultiset of $Q'$ consisting of those $s\in Q'$ for which $rs\in \sum_{J\s I^c,J\neq \emptyset} V_J(s)\otimes \F_2^{I^c\setminus J}$. Then we have $|Q'(r)|\geq (1-g_3)|Q'|$.
	
	\smallskip
	
	Note that we can write $r=r_1+r_2+r_3+r_4$ where $r_1\in \sum_{J\subset I,J\neq I,J\neq \emptyset} W_J\otimes \F^{J^c}$, $r_2\in \sum_{J\prec I, J\not \subset I} (W_J\otimes \F^{J^c}+\F^{J}\otimes W_{J^c})+\sum_{J\succ I} \F^{J}\otimes H_{J^c}(r)$, $r_3\in W_{\lbrack d\rbrack}+\sum_{J\subset I,J\neq I,J\neq \emptyset} \F^{J}\otimes W_{J^c}$ and $r_4\in \F^{I}\otimes H_{I^c}(r)$. By (1) in Lemma \ref{claim}, the elements of $Q'$ belong to $\bigcap_{J\subset I,J\neq I,J\neq \emptyset}(W_J^{\perp}\otimes \F^{I \setminus J})$, so we have $r_1s=0$ for every $s\in Q'$. Since $\dim(W_J),\dim(W_{J^c}),\dim(H_{J^c}(r))\leq k$, $r_2s$ is $2^{d}k$-degenerate. Also, $r_3s\in \sum_{J\subset I,J\neq I} ((\F^{J}\otimes W_{J^c})s)$. (Here and below, for a subspace $L\s \cG$ and an array $s\in \F^{I}$, we write $Ls$ for the subspace $\{rs: r\in L\}\s \F^{I^c}$.) It follows that for every $s\in Q'(r)$, there exists some $t(s)\in V_{I^c}(s)+\sum_{J\subset I,J\neq I} ((\F^{J}\otimes W_{J^c})s)$ such that $r_4s-t(s)$ is $g_4$-degenerate for $g_4=g_1+2^dk$ (we have used that $\dim (V_{J}(s))\leq g_1$).
	
	Now let $g_5=\frac{g_3}{|\F|^{k}}$. Let $X$ be the subset of $\F^{I^c}$ consisting of those arrays $x$ for which for at least $g_5|Q'|$ choices $s\in Q'$, there exists some $t(s)\in V_{I^c}(s)+\sum_{J\subset I,J\neq I}((\F^{J}\otimes W_{J^c})s)$ such that $x-t(s)$ is $g_4$-degenerate. Choose a maximal subset $\{x_1,\dots,x_m\}\s X$ such that no $x_i-x_j$ with $i\neq j$ is $2g_4$-degenerate. Then there do not exist $i\neq j$, $s\in Q'$ and $t\in V_{I^c}(s)+\sum_{J\subset I,J\neq I} ((\F^{J}\otimes W_{J^c})s)$ with $x_i-t$ and $x_j-t$ both $g_4$-degenerate. It follows, by the definition of $X$, and using that the dimension of $V_{I^c}(s)+\sum_{J\subset I,J\neq I} ((\F^{J}\otimes W_{J^c})s)$ is at most $g_1+2^dk$ that $mg_5|Q'|\leq |Q'|\cdot |\F|^{g_1+2^dk}$, therefore $m\leq g_6=\frac{|\F|^{g_1+2^dk}}{g_5}$. Let $Z=\text{span}(x_1,\dots,x_m)$. Then $\dim(Z)\leq g_6$ and for every $x\in X$, there is some $z\in Z$ such that $x-z$ is $2g_4$-degenerate.
	
	By the definition of $X$, if $t\not \in X$, then the number of choices $s\in Q'(r)$ for which $r_4s=t$ is at most $g_5|Q'|$. On the other hand, notice that $r_4s\in H_{I^c}(r)$ for every $s\in Q'$. Since $|H_{I^c}(r)|\leq |\F|^{k}$, it follows that $r_4s\in X$ for at least $|Q'(r)|-|\F|^kg_5|Q'|=|Q'(r)|-g_3|Q'|\geq (1-2g_3)|Q'|=(1-2^{-3^{d+2}})|Q'|$ choices $s\in Q'$.
	
	Let $X(r)=X\cap H_{I^c}(r)$. Let $t_1,\dots,t_{\a}$ be a maximal linearly independent subset of $X(r)$ and extend it to a basis $t_1,\dots,t_{\a},t'_1,\dots,t'_{\beta}$ for $H_{I^c}(r)$. Now if a linear combination of $t_1,\dots,t_{\a},t'_1,\dots,t'_{\beta}$ is in $X$, then the coefficients of $t'_1,\dots,t'_{\beta}$ are all zero. Write $r_4=\sum_{i\leq \a}s_i\otimes t_i+\sum_{j\leq \beta}s'_j\otimes t'_j$ for some $s_i,s'_j\in \F^{I}$. Since $r_4q\in X$ for at least $(1-2^{-3^{d+2}})|Q'|$ choices $q\in Q'$, we have, for all $j$, that $s'_j.q=0$ for at least $(1-2^{-3^{d+2}})|Q'|$ choices $q\in Q'$. Thus, by (2) in Lemma \ref{claim} there exist subspaces $L_J\s \F^{J}$ ($J\s I,J\neq \emptyset$) not depending on $r$, and of dimension at most $\tower_{8|\F|}((d+3)^{|I|+4}+2,k)$ such that $s'_j\in \sum_{J\s I,J\neq \emptyset} L_J\otimes \F^{I\setminus J}$ for all $j$. Thus, $r_4\in \sum_{i\leq \a}s_i\otimes t_i+\sum_{J\s I,J\neq \emptyset} L_J\otimes \F^{J^c}$. Moreover, for every $i\leq \a$, we have $t_i\in X$, so there exist $z_i\in Z$ such that $t_i-z_i$ is $2g_4$-degenerate. It follows that $r_4\in \F^{I}\otimes Z+\sum_{J\supset I,J\neq I,J\subset \lbrack d-1\rbrack} \F^{J}\otimes K'_{J^c}(r)+\sum_{J\s I,J\neq \emptyset} L_J\otimes \F^{J^c}$ for some $K'_{J^c}(r)\s \F^{J^c}$ of dimension at most $\a\cdot 2g_4\leq k\cdot 2g_4$.
	
	We claim that $\dim(Z),\dim(K'_{J^c})$ and $\dim(L_J)$ are all bounded by $\tower_{8|\F|}((d+3)^{|I|+4}+2,k)$. If this holds, then the proof of this lemma is complete, since $k+\tower_{8|\F|}((d+3)^{|I|+4}+2,k)\leq \tower_{8|\F|}((d+3)^{|I|+4}+3,k)$.
	
	Now $\dim(K'_{J^c})\leq 2kg_4=2k(g_1+2^dk)\leq 2k(\tower_{8|\F|}(4,k)+2^dk)\leq \tower_{8|\F|}((d+3)^{|I|+4}+2,k)$.
	
	Also, $\dim(Z)\leq g_6=\frac{|\F|^{g_1+2^dk}}{g_5}=2^{3^{d+2}+1}|\F|^{g_1+2^dk+k}\leq |\F|^{2g_1}\leq \tower_{8|\F|}((d+3)^{|I|+4}+2,k)$.
	
	Finally, as we have already noted, $\dim(L_J)\leq \tower_{8|\F|}((d+3)^{|I|+4}+2,k)$. This completes the proof of the claim and the lemma.
\end{proof}

\begin{proof}[Proof of Lemma \ref{mainlemma}]
	As stated earlier, the proof goes by induction on $d$. For $d=1$, by Bogolyubov's lemma there is a subspace $U\s \F^{n_1}$ of codimension at most $1/\d^2$ contained in $2\cB'-2\cB'$. Choose $Q=U$. Now if $r.q=0$ for at least $(1-2^{-3^4})|Q|$ choices $q\in Q$ then the same holds for all $q\in Q$, therefore $r\in U^{\perp}$, but $\dim(U^{\perp})\leq 1/\d^2\leq \tower_{8|\F|}(5^5,1/\d)$, so the case $d=1$ is proved.
	
	\smallskip  
	
	Now let us assume that $d\geq 2$. Extend the total order $\prec$ defined above such that it now contains $\emptyset$ which has $\emptyset \prec I$ for every non-empty $I\s \lbrack d-1\rbrack$. Say $\emptyset=I_0\prec I_1\prec I_2\prec \dots \prec I_{2^{d-1}-1}$ where $\{I_0,\dots,I_{2^{d-1}-1}\}=P(\lbrack d-1\rbrack)$.
	
	\smallskip
	
	\noindent \emph{Claim.} For every $0\leq i\leq 2^{d-1}-1$ there exists a multiset $Q_{I_i}$ with elements chosen from $2^{3^{d+3}}\cB'-2^{3^{d+3}}\cB'$, and subspaces $W_{I_j}(i)\s \F^{I_j}$, $W_{(I_j)^c}(i)\s \F^{(I_j)^c}$ for every $j\leq i$ (for $j=0$, we only require $W_{\lbrack d\rbrack}(i)$ and not $W_{\emptyset}(i)$) of dimension at most $$g_1(i)=\tower_{8|\F|}\Big((d+3)^{d+3}+3+\sum_{1\leq j\leq i} ((d+3)^{|I_j|+4}+3),1/\d \Big),$$ with the following property. If $r\in \cG$ has $r.q=0$ for at least $(1-\frac{1}{4}(2^{-3^{d+2}})^2)|Q_{I_j}|$ choices $q\in Q_{I_j}$ for all $j\leq i$, then $r\in W_{\lbrack d\rbrack}(i)+ \sum_{1\leq j\leq i} (W_{I_j}(i)\otimes \F^{(I_j)^c}+\F^{I_j}\otimes W_{(I_j)^c}(i))+\sum_{j>i}\F^{I_j}\otimes H_{(I_j)^c}(i,r)$ holds for some $H_{(I_j)^c}(i,r)$ possibly depending on $r$ and of dimension at most $g_1(i)$.
	
	\smallskip
	
	\noindent \emph{Proof of Claim.} This is proved by induction on $i$. For $i=0$, by Lemma \ref{focusondeg}, there exist $Q_{\emptyset}\s 2\cB'-2\cB'$ and $V_{\lbrack d\rbrack}\s \F^{\lbrack d\rbrack}$ of dimension at most $5|\F|^{4/\d^2}$ such that if $r.q=0$ for at least $\frac{3}{4}|Q_{\emptyset}|$ choices $q\in Q_{\emptyset}$, then $r$ can be written as $r=x+y$ where $x\in V_{\lbrack d\rbrack}$ and $y$ is $g_2$-degenerate for $g_2=\tower_{8|\F|}((d+3)^{d+3}+3,1/\d)$. Hence we can take $W_{\lbrack d\rbrack}(0)=V_{\lbrack d\rbrack}$.
	
	Once we have found suitable sets $W_{I_j}(i-1)$ and $W_{(I_j)^c}(i-1)$ for all $j\leq i-1$, we can apply Lemmas \ref{claim} and \ref{keylemma} with $I=I_i$ and $k=g_1(i-1)$ to find a suitable $Q_{I_i}$, $W_{I_j}(i)$ and $W_{(I_j)^c}(i)$ for all $j\leq i$, and the claim is proved, since $g_1(i)=\tower_{8|\F|}((d+3)^{|I_i|+4}+3,g_1(i-1))$.
	
	\smallskip
	
	Now, after taking several copies of each $Q_{I}$, we may assume that additionally $\max_I |Q_I|\leq 2\min_{I} |Q_I|$. Let $Q=\bigcup_{I\s \lbrack d-1\rbrack}Q_I$ and suppose that $r.q=0$ for at least $(1-2^{-3^{d+3}})|Q|$ choices $q\in Q$. Since $2^{-3^{d+3}}\leq \frac{1}{2\cdot 2^{d-1}}\cdot \frac{1}{4}(2^{-3^{d+2}})^2$, it follows that for every $I\s \lbrack d-1\rbrack$ we have $r.q=0$ for at least $(1-\frac{1}{4}(2^{-3^{d+2}})^2)|Q_I|$ choices $q\in Q_I$. By the Claim with $i=2^{d-1}-1$, we get that $r\in \sum_{I\s \lbrack d\rbrack} V_I\otimes \F^{I^c}$ for some $V_I\s \F^{I}$ not depending on $r$, and of dimension at most $$g_1(2^{d-1}-1)=\tower_{8|\F|}\Big((d+3)^{d+3}+3+\sum_{\emptyset \neq I\s \lbrack d-1\rbrack} ((d+3)^{|I|+4}+3),1/\d \Big).$$ But
	\begin{align*}
		 (d+3)^{d+3}+3+\sum_{\emptyset \neq I\s \lbrack d-1\rbrack} ((d+3)^{|I|+4}+3) & \leq (d+3)^{d+3}+3\cdot2^{d-1}+(d+3)^4\sum_{1\leq k\leq d-1} {d-1 \choose k} (d+3)^k \\
		 & \leq (d+3)^{d+3}+3\cdot 2^{d-1}+(d+3)^4(d+4)^{d-1} \\
		 & \leq (d+3)^{d+3}+3\cdot 2^{d-1}+(d+4)^{d+3} \\
		 & \leq (d+4)^{d+4}
	\end{align*}
	This completes the proof of the lemma.
\end{proof}

\section*{Acknowledgment} I would like to thank Timothy Gowers for helpful discussions. I am also grateful for his valuable comments on the paper.

\end{document}